%% file: Article.tex
\documentclass[12pt]{article}
\usepackage[english]{babel}
\usepackage{ucs}
\usepackage[latin1]{inputenc}
\usepackage[T1]{fontenc}
\usepackage{amsmath,amsfonts,amsthm,enumerate}
\usepackage{amssymb}
\usepackage{geometry}
\usepackage{verbatim}
\usepackage{graphicx}
\usepackage[all]{xy}
\usepackage{bm}
\usepackage{pdfsync}
\usepackage{tabularx}
\usepackage[toc,page]{appendix}
\usepackage[babel=true]{csquotes}
\usepackage{xcolor}
\usepackage{textcomp}
\usepackage{cancel}
\usepackage{float}

\newcommand*{\leftmapsto}{\mathbin{\reflectbox{$\mapsto$}}}

\newtheoremstyle{def}{2em}{2em}{\normalfont}{}{\bfseries}{}{0.5em}{}
\newtheoremstyle{th}{2em}{1em}{\itshape}{}{\bfseries}{}{0.5em}{}
\newtheoremstyle{rem}{2em}{2em}{\normalfont}{}{\bfseries}{~:}{0.5em}{}

\theoremstyle{def}
\newtheorem{defin}{Definition}

\theoremstyle{th}
\newtheorem{prop}{Proposition}
\newtheorem{theo}{Theorem}
\newtheorem*{theo*}{Theorem}

\newtheorem{cor}{Corollary}

\theoremstyle{rem}
\newtheorem{rem}{Remark}

\begin{document}
\allowdisplaybreaks


\input{Introduction}

\input{Partie1courte}

\input{Partie2}

\input{Partie3}

\input{Partie4}

\input{Partie5}

\end{document}

%% file: Introduction.tex
\begin{center}
\begin{LARGE} Constructions of minimal periodic surfaces and minimal annuli in Sol$_{3}$ \end{LARGE}

\vspace{2\baselineskip}

Christophe \textsc{Desmonts}
\end{center}

\begin{abstract}
We construct two one-parameter families of minimal properly embedded surfaces in the Lie group Sol$_{3}$ using a Weierstrass-type representation. These surfaces are not invariant by a one-parameter group of ambient isometries. The first one can be viewed as a family of helicoids, and the second one is a family of minimal annuli called catenoids. Finally we study limits of these catenoids, and in particular we show that one of these limits is a new minimal entire graph.
\end{abstract}

\section{Introduction}

The aim of this paper is to construct two one-parameter families of examples of properly embedded minimal surfaces in the Lie group Sol$_{3}$, endowed with its standard metric. This group is a homogeneous Riemannian manifold with a $3$-dimensional isometry group and is one of the eight Thurston's geometries. There is no rotation in Sol$_{3}$, and so no surface of revolution.

The Hopf differential, which exists on surfaces in every $3$-dimensional space form, has been generalized by Abresch and Rosenberg to every $3$-dimensional homogeneous Riemannian manifolds with $4$-dimensional isometry group (see \cite{AbreschRosenberg1} and \cite{AbreschRosenberg2}). This tool leads to a lot of researches in the field of CMC surfaces in Nil$_{3}$, $\widetilde{PSL_{2}}(\mathbb{R})$ and in the Berger spheres. More precisely, Abresch and Rosenberg proved that the generalized Hopf differential exists in a simply connected Riemannian 3-manifold if and only if its isometry group has at least dimension $4$ (see \cite{AbreschRosenberg2}).

\begin{sloppypar}
Berdinskii and Taimanov gave a representation formula for minimal surfaces in $3$-dimensional Lie groups in terms of spinors, but they pointed out some difficulties for using this theory in the case of Sol$_{3}$ (see \cite{BerdinskiiTaimanov}). Nevertheless, Daniel and Mira mentioned some basic examples of minimal graphs in Sol$_{3}$ : $x_{1} = ax_{2} + b$, $x_{1} = ae^{-x_{3}}$, $x_{1} = ax_{2}e^{-x_{3}}$ and $x_{1} = x_{2}e^{-2x_{3}}$ (and their images by isometries of Sol$_{3}$ : see \cite{DanielMira}). Finally, Ana Menezes constructed singly and doubly periodic Scherk minimal surfaces in Nil$_{3}$ and Sol$_{3}$ (see \cite{Menezes}).
\end{sloppypar}

Thus, the method that we use in this paper is the one used by Daniel and Hauswirth (see \cite{DanielHauswirth}) in Nil$_{3}$ to construct minimal embedded annuli : we construct a first one-parameter family of embedded minimal surfaces called \emph{helicoids} and we calculate its Gauss map $g$. A result of Inoguchi and Lee (see \cite{Inoguchi}) claims that this map is harmonic for a certain metric on $\overline{\mathbb{C}}$. Then we seek another family of maps $g$ with separated variables that still satisfies the harmonic equation, and we use a Weierstrass type representation given by Inoguchi and Lee to construct a minimal immersion whose Gauss map is $g$. We prove that these immersions are periodic, so we get minimal annuli. As far as I know, these annuli are the first examples of non simply connected surfaces with finite topology (that is, diffeomorphic to a compact surface without a finite number of points) in Sol$_{3}$.

The model we use for Sol$_{3}$ is described in Section $2$. In the third section we give some properties of the Gauss map of a conformal minimal immersion in Sol$_{3}$ (see \cite {DanielMira}). In the fourth section we construct the family $(\mathcal{H}_{K})_{K \in ]-1;1[}$ of helicoids, and finally we construct the family $(\mathcal{C}_{\alpha})_{\alpha \in ]-1;1[ \backslash \{ 0 \}}$ of embedded minimal annuli. The study of the limit case of the parameter of this family gives an other example of minimal surface in Sol$_{3}$, which is an entire graph. None of these surfaces is invariant by a one-parameter family of isometries.

\begin{theo*}
There exists a one-parameter family $(\mathcal{C}_{\alpha})_{\alpha \in ]-1;1[ \backslash \{ 0 \}}$ of properly embedded minimal annuli in \emph{Sol}$_{3}$, called \emph{catenoids}, having the following properties :
\begin{enumerate}
\item The intersection of $\mathcal{C}_{\alpha}$ and any plane $\{ x_{3} = \lambda \}$ is a non-empty closed embedded convex curve ;
\item The annulus $\mathcal{C}_{\alpha}$ is conformally equivalent to $\mathbb{C}~\backslash \{ 0 \}$ ;
\item The annulus $\mathcal{C}_{\alpha}$ has $3$ symmetries fixing the origin :
\begin{itemize}
\item rotation of angle $\pi$ around the $x_{3}$-axis ;
\item reflection with respect to $\{ x_{1} = 0 \}$ ;
\item reflection with respect to $\{ x_{2} = 0 \}$.
\end{itemize}
\end{enumerate}
\end{theo*}

%% file: Partie1courte.tex
\section{The Lie group Sol$_{3}$}

\begin{defin}
 The Lie group Sol$_{3}$ is the set $\mathbb{R}^{3}$ with the multiplication $\ast$ defined by

$$(x_{1},x_{2},x_{3}) \ast (y_{1},y_{2},y_{3}) = (y_{1}e^{-x_{3}} + x_{1},y_{2}e^{x_{3}} + x_{2},x_{3}+y_{3})$$
for all $(x_{1},x_{2},x_{3}), (y_{1},y_{2},y_{3}) \in \mathbb{R}^{3}$. The identity element for this law is $0$ and the inverse element of $(x_{1},x_{2},x_{3})$ is $(x_{1},x_{2},x_{3})^{-1} = (-x_{1}e^{x_{3}},-x_{2}e^{-x_{3}},-x_{3})$. The law is not commutative.
\end{defin}

The left multiplication $l_{a}$ by an element $a = (a_{1},a_{2},a_{3}) \in \mathbb{R}^{3}$ is given for all $x = (x_{1},x_{2},x_{3}) \in \mathbb{R}^{3}$ by

$$\begin{array}{rcl}
   l_{a}(x) = a \ast x & = & (x_{1}e^{-a_{3}} + a_{1},x_{2}e^{a_{3}} + a_{2},a_{3}+x_{3}) \\
		       		   & = & a + M_{a}x,
  \end{array}$$
where

$$M_{a} = \begin{pmatrix}
      e^{-a_{3}} & 0 & 0 \\
      0 & e^{a_{3}} & 0 \\
      0 & 0 & 1
      \end{pmatrix}.$$
For the metric $(\cdot,\cdot)$ on Sol$_{3}$ to be left-invariant, it has to satisfy

$$(M_{a}X,M_{a}Y)_{a \ast x} = (X,Y)_{x}$$
for all $a, x, X, Y \in \mathbb{R}^{3}$. We define a left-invariant riemannian metric for $x, X, Y \in \mathbb{R}^{3}$ by the formula

\begin{eqnarray}
 (X,Y)_{x} = \left\langle M_{x^{-1}}X,M_{x^{-1}}Y \right\rangle, \label{eq1}
\end{eqnarray}
where $\langle \cdot,\cdot \rangle$ is the canonical scalar product on $\mathbb{R}^{3}$ and $x^{-1}$ is the inverse element of $x$ in Sol$_{3}$. The formula \eqref{eq1} leads to the following expression of the previous metric

\begin{eqnarray}
 ds^{2}_{x} = e^{2x_{3}} dx_{1}^{2} + e^{-2x_{3}} dx_{2}^{2} + dx_{3}^{2}, \label{eq2}
\end{eqnarray}
where $(x_{1},x_{2},x_{3})$ are canonical coordinates of $\mathbb{R}^{3}$. Since the translations are isometries now, Sol$_{3}$ is a homogeneous manifold with this metric.

\begin{rem}
This metric is not the only left-invariant one which can be put on Sol$_{3}$. In fact, there exists a two-parameter family of non isometric left-invariant metrics on Sol$_{3}$. One of these parameters is a homothetic one, and the other metrics than \eqref{eq2} (and the ones which are homothetic to \eqref{eq2}) have no reflections ; see \cite{MeeksPerez}.
\end{rem}

By setting

$$E_{1}(x) = e^{-x_{3}} \partial_{1},~~~~~~E_{2}(x) = e^{x_{3}} \partial_{2},~~~~~~E_{3}(x) = \partial_{3},$$
we obtain a left-invariant orthonormal frame $(E_{1},E_{2},E_{3})$. Thus, we have now two frames to express the coordinates of a vector field on Sol$_{3}$ ; we will denote into brackets the coordinates in the frame $(E_{1},E_{2},E_{3})$ ; then we have at a point $x \in \mbox{Sol}_{3}$

\begin{eqnarray}
a_{1} \partial_{1} + a_{2} \partial_{2} + a_{3} \partial_{3} = \left( \begin{matrix} a_{1} \\ a_{2} \\ a_{3} \end{matrix} \right) = \left[ \begin{matrix} e^{x_{3}} a_{1} \\ e^{-x_{3}} a_{2} \\ a_{3} \end{matrix} \right]. \label{eq3}
\end{eqnarray}
The following property holds (cf \cite{DanielMira}) :

\begin{prop}
\label{propisometriessol}
 The isotropy group of the origin of \emph{Sol}$_{3}$ is isomorphic to the dihedral group $\mathcal{D}_{4}$ and generated by the two isometries

$$\sigma : (x_{1},x_{2},x_{3}) \longmapsto (x_{2},-x_{1},-x_{3})~~~~~~\mbox{and}~~~~~~\tau : (x_{1},x_{2},x_{3}) \longmapsto (-x_{1},x_{2},x_{3}).$$
These two isometries are orientation-reversing, and $\sigma$ has order $4$ and $\tau$ has order $2$.
\end{prop}

For instance, the planar reflection with respect to $\{ x_{2} = 0 \}$ is given by $\sigma^{2} \tau : (x_{1},x_{2},x_{3}) \longmapsto (x_{1},-x_{2},x_{3})$. The set of isotropies of $(0,0,0)$ is

$$\begin{array}{rcl}
(x_{1},x_{2},x_{3}) & \longmapsto & (x_{1},x_{2},x_{3})~; \\
(x_{1},x_{2},x_{3}) & \overset{\sigma}{\longmapsto} & (x_{2},-x_{1},-x_{3}) \mbox{ : rotation around } E_{3} \mbox{ of angle } 3\pi / 2 \mbox{ composed with the} \\
				    &  & ~~~~~~~~~~~~~~~~~~~ \mbox{ reflection with respect to } \{x_{3} = 0 \}~; \\
(x_{1},x_{2},x_{3}) & \overset{\tau}{\longmapsto} & (-x_{1},x_{2},x_{3}) \mbox{ : reflection with respect to } \{x_{1} = 0 \}~;\\
(x_{1},x_{2},x_{3}) & \overset{\sigma^{2}}{\longmapsto} & (-x_{1},-x_{2},x_{3}) \mbox{ : rotation of angle } \pi \mbox{ around } E_{3}~; \\
(x_{1},x_{2},x_{3}) & \overset{\sigma^{3}}{\longmapsto} & (-x_{2},x_{1},-x_{3}) \mbox{ : rotation around } E_{3} \mbox{ of angle } \pi / 2 \mbox{ composed with the} \\
				    &  & ~~~~~~~~~~~~~~~~~~~ \mbox{ reflection with respect to } \{x_{3} = 0 \}~; \\
(x_{1},x_{2},x_{3}) & \overset{\sigma \tau}{\longmapsto} & (x_{2},x_{1},-x_{3}) \mbox{ : rotation of angle } \pi \mbox{ around the axis } \{(x_{1},x_{1},0)\}~; \\
(x_{1},x_{2},x_{3}) & \overset{\sigma^{2} \tau}{\longmapsto} & (x_{1},-x_{2},x_{3}) \mbox{ : reflection with respect to } \{x_{2} = 0 \}~; \\
(x_{1},x_{2},x_{3}) & \overset{\sigma^{3} \tau}{\longmapsto} & (-x_{2},-x_{1},-x_{3}) \mbox{ : rotation of angle } \pi \mbox{ around the axis } \{(x_{1},-x_{1},0)\}.
\end{array}$$
We deduce the following theorem :

\begin{theo}
The isometry group of \emph{Sol}$_{3}$ has dimension $3$.
\end{theo}

Finally, we express the Levi-Civita connection $\nabla$ of Sol$_{3}$ associated to the metric given by the equation \eqref{eq2} in the frame $(E_{1},E_{2},E_{3})$. In a first time, we calculate the Lie brackets of the vectors of the frame. The Lie bracket in the Lie algebra $\mathfrak{sol}_{3}$ of Sol$_{3}$ is given by

$$[X,Y] = (Y_{3}X_{1} - X_{3}Y_{1},X_{3}Y_{2} - Y_{3}X_{2},0)$$

for all $X=(X_{1},X_{2},X_{3})$ and $Y = (Y_{1},Y_{2},Y_{3})$. Then we have

$$[E_{1},E_{2}] = 0,~~~~~~[E_{1},E_{3}] = E_{1},~~~~~~[E_{2},E_{3}] = -E_{2}.$$
We deduce the following results :

\begin{eqnarray*}
 \nabla_{E_{1}}E_{1} = -E_{3}, & ~~~~~~~~\nabla_{E_{2}}E_{1} = 0,~~~~~~~~ & \nabla_{E_{3}}E_{1} = 0,\\
 \nabla_{E_{1}}E_{2} = 0, & \nabla_{E_{2}}E_{2} = E_{3}, & \nabla_{E_{3}}E_{2} = 0, \\
 \nabla_{E_{1}}E_{3} = E_{1}, & \nabla_{E_{2}}E_{3} = -E_{2}, & \nabla_{E_{3}}E_{3} = 0.
\end{eqnarray*}

%% file: Partie2.tex
\section{The Gauss map}

Let $\Sigma$ be a riemannian oriented surface and $z = u+iv$ an isothermal system of coordinates in $\Sigma$. Let $x : \Sigma \longrightarrow \mbox{Sol}_{3}$ be a conformal immersion. We set

$$x = \begin{pmatrix} x_{1} \\ x_{2} \\ x_{3} \end{pmatrix}$$
and we define $\lambda \in \mathbb{R}^{*}_{+}$ by

$$2(x_{z},x_{\overline{z}}) = \left\| x_{u} \right\|^{2} = \left\| x_{v} \right\|^{2} = \lambda.$$
Thus, a unit normal vector field is $N : \Sigma \longrightarrow T \mbox{Sol}_{3}$ defined by 

\begin{eqnarray*}
N & = & -\frac{2i}{\lambda} x_{z} \wedge x_{\overline{z}} \\
  & := & \left[ \begin{matrix} N_{1} \\ N_{2} \\ N_{3} \end{matrix} \right].
\end{eqnarray*}
Hence we define

$$\widehat{N} : \Sigma \longrightarrow \mathbb{S}^{2} \subset \mathbb{R}^{3}$$
by the formula

$$M_{x^{-1}} N = \widehat{N},$$
that is

\begin{eqnarray*}
\widehat{N} & = & \begin{pmatrix} e^{x_{3}} & 0 & 0 \\ 0 & e^{-x_{3}}  & 0 \\ 0 & 0 & 1 \end{pmatrix} \begin{pmatrix} N_{1} e^{-x_{3}} \\ N_{2} e^{x_{3}} \\ N_{3} \end{pmatrix} \\
            & = & \begin{pmatrix} N_{1} \\ N_{2} \\ N_{3} \end{pmatrix}.
\end{eqnarray*}

\begin{defin}
The Gauss map of the immersion $x$ is the application

$$g = \sigma \circ \widehat{N} : \Sigma \longrightarrow \mathbb{C} \cup \{ \infty \} = \overline{\mathbb{C}},$$
where $\sigma$ is the stereographic projection with respect to the southern pole, i.e.

\begin{eqnarray}
 N & = & \dfrac{1}{1 + \left| g \right|^{2}} \left[ \begin{matrix} 2 \Re{(g)} \\ 2 \Im{(g)} \\ 1 - \left| g \right|^{2} \end{matrix} \right] \label{eq4} ; \\
 g & = & \dfrac{N_{1} + iN_{2}}{1 + N_{3}} \label{eq5}.
\end{eqnarray}

\end{defin}

The following result can be found in \cite{Inoguchi}. It can be viewed as a Weierstrass representation in Sol$_{3}$.

\begin{theo}
\label{theoweierstrasssol3}
Let $x : \Sigma \longrightarrow \emph{Sol}_{3}$ be a conformal minimal immersion and $g : \Sigma \longrightarrow \overline{\mathbb{C}}$ its Gauss map. Then, whenever $g$ is neither real nor pure imaginary, it is nowhere antiholomorphic ($g_{z} \neq 0$ for every point for any local conformal parameter $z$ on $\Sigma$), and it satisfies the second order elliptic equation

\begin{eqnarray}
 g_{z \overline{z}} = \dfrac{2g g_{z} g_{\overline{z}}}{g^{2} - \overline{g}^{2}}. \label{eq6}
\end{eqnarray}
Moreover, the immersion $x = (x_{1},x_{2},x_{3})$ can be expressed in terms of $g$ by the representation formulas

\begin{eqnarray}
 {x_{1}}_{z} = e^{-x_{3}} \dfrac{(\overline{g}^{2} - 1)g_{z}}{g^{2} - \overline{g}^{2}},~~~~{x_{2}}_{z} = ie^{x_{3}} \dfrac{(\overline{g}^{2} + 1)g_{z}}{g^{2} - \overline{g}^{2}},~~~~{x_{3}}_{z} = \dfrac{2\overline{g}g_{z}}{g^{2} - \overline{g}^{2}} \label{eq7}
\end{eqnarray}
whenever it is well defined.

Conversely, given an application $g : \Sigma \longrightarrow \overline{\mathbb{C}}$ defined on a simply connected riemannian surface $\Sigma$ that verify the equation \eqref{eq6}, then the application $x : \Sigma \longrightarrow \emph{Sol}_{3}$ given by the representation formulas \eqref{eq7} is a conformal minimal immersion with possibly branched points whenever it is well defined, whose Gauss map is $g$.
\end{theo}

\begin{rem}
\label{remarqueequationharmonique}
\begin{enumerate}
\item There exists a similar result for the case of CMC $H$ surfaces ; see \cite{DanielMira}.
\item The equation \eqref{eq6} is the one of harmonic maps $g : \Sigma \longrightarrow (\overline{\mathbb{C}},ds^{2})$, with $ds^{2}$ the metric on $\overline{\mathbb{C}}$ defined by

$$ds^{2} = \dfrac{\left| d \omega \right|^{2}}{\left| \omega^{2} - \overline{\omega}^{2} \right|}.$$
This is a singular metric, not defined on the real and pure imaginary axes. See \cite{Inoguchi} for more details.
\item The equation \eqref{eq6} can be only considered at points where $g \neq \infty$. But if $g$ is a solution of \eqref{eq6}, $i/g$ is one too at points where $g \neq 0$. The conjugate map $\overline{g}$ and every $g \circ \phi$, with $\phi$ a locally injective holomorphic function, are solutions too. Moreover, if $g$ is a solution nowhere antiholomorphic of \eqref{eq6}, and $x$ the induced conformal minimal immersion, then $ig$ and $1 / g$ induce conformal minimal immersions given by $\sigma \circ x$ and $\tau \circ x$. Finally, $\overline{g}$ is the Gauss map of $\sigma^{2} \tau \circ x$ after a change of orientation.
\end{enumerate}
\end{rem}

\begin{defin}
The Hopf differential of the map $g$ is the quadratic form

$$Q = q dz^{2} = \dfrac{g_{z} \overline{g}_{z}}{g^{2} - \overline{g}^{2}} dz^{2}.$$
\end{defin}

\begin{rem}
\begin{enumerate}
\item The function $q$ depends on the choice of coordinates, whereas $Q$ does not.
\item As we said in the introduction of this paper, the Hopf differential (or its Abresch-Rosenberg generalization) is not defined in Sol$_{3}$. If we apply the definition of the Hopf differential of the harmonic maps in $(\overline{\mathbb{C}},ds^{2})$, we get

$$Q = \dfrac{g_{z} \overline{g}_{z}}{|g^{2} - \overline{g}^{2}|} dz^{2},$$
but this one leads to a non smooth differential. Because $g^{2} - \overline{g}^{2}$ is pure imaginary, the definitions are related on each quarter of the complex plane by multiplication by $i$ or $-i$, depending on which quarter we are. Thus, this "Hopf differential" is defined and holomorphic only on each of the four quarter delimited by the real and pure imaginary axes.
\end{enumerate}
\end{rem}

%% file: Partie3.tex
\section{Construction of the helicoids in Sol$_{3}$}

In this section we construct a one-parameter family of helicoids in Sol$_{3}$ : we call helicoid a minimal surface containing the $x_{3}$-axis whose intersection with every plane $\{ x_{3} = \mbox{cte} \}$ is a straight line and which is invariant by left multiplication by an element of the form $(0,0,T)$ for some $T \neq 0$ :

\begin{theo}
There exists a one-parameter family $(\mathcal{H}_{K})_{K \in ]-1;1[ \backslash \{ 0 \}}$ of properly embedded minimal helicoids in \emph{Sol}$_{3}$ having the following properties :
\begin{enumerate}
\item For all $K \in ]-1;1[ \backslash \{ 0 \}$, the surface $\mathcal{H}_{K}$ contains the $x_{3}$-axis ;
\item For all $K \in ]-1;1[ \backslash \{ 0 \}$, the intersection of $\mathcal{H}_{K}$ and any horizontal plane $\{ x_{3} = \lambda \}$ is a straight line ;
\item For all $K \in ]-1;1[ \backslash \{ 0 \}$, there exists $T_{K}$ such that $\mathcal{H}_{K}$ is invariant by left multiplication by $(0,0,T_{K})$ ;
\item The helicoids $\mathcal{H}_{K}$ have $3$ symmetries fixing the origin :
\begin{itemize}
\item rotation of angle $\pi$ around the $x_{3}$-axis ;
\item rotation of angle $\pi$ around the $(x,x,0)$-axis ;
\item rotation of angle $\pi$ around the $(x,-x,0)$-axis.
\end{itemize}
\end{enumerate}
\end{theo}

Let $K \in~]-1,1[$ ; we define a map $g : \mathbb{C} \longrightarrow \overline{\mathbb{C}}$ by

$$g(z=u+iv) = e^{-u} e^{ib(v)} e^{-i\pi/4},$$
where $b$ satisfies the ODE

\begin{eqnarray}
 b' = \sqrt{1 - K\cos{(2b)}},~~b(0)=0. \label{eq8}
\end{eqnarray}

\begin{prop}
The map $b$ is well defined and has the following properties :
\begin{enumerate}
\item The function $b$ is an increasing diffeomorphism from $\mathbb{R}$ onto $\mathbb{R}$ ;
\item The function $b$ is odd ;
\item There exists a real number $W > 0$ such that

$$\forall v \in \mathbb{R},~~~b(v+W) = b(v) + \pi~;$$
\item The function $b$ satisfies $b(kW) = k \pi$ for all $k \in \mathbb{Z}$.
\end{enumerate}
\end{prop}

\begin{proof}[Proof]
Since $K \in~]-1,1[$, there exists $r > 0$ such that $1 - K\cos{(2b)} \in~]r,2[$ ; the Cauchy-Lipschitz theorem can be applied, and $b$ is well defined.
\begin{enumerate}
\item By \eqref{eq8} we have $b' > 0$ on its definition domain, and $\sqrt{r} < b' < 2$. So $b'$ is bounded by two positive constants, then $b$ is defined on the entire $\mathbb{R}$, and

$$\lim\limits_{v \to \pm \infty} b(v) = \pm \infty.$$
\item The function $\widehat{b} : v \longmapsto -b(-v)$ satisfies \eqref{eq8} with $\widehat{b}(0) = 0$ ; hence $\widehat{b} = b$ and $b$ is odd.

\item There exists $W > 0$ such that $b(W) = \pi$ ; Then the function $\widetilde{b} : v \longmapsto b(v+W) - \pi$ satisfies \eqref{eq8} with $\widetilde{b}(0) = 0$ ; hence $\widetilde{b} = b$.

\item Obvious.
\end{enumerate}
\end{proof}

\begin{cor}
\label{corollaireb}
We have $b(kW/2) = k \pi / 2$ for all $k \in 2 \mathbb{Z} +1$.
\end{cor}

\begin{proof}[Proof]

We have
\begin{eqnarray*}
b \left( \dfrac{W}{2} \right) & = & b \left( -\dfrac{W}{2} + W \right) \\
		& = & - b \left( \dfrac{W}{2} \right) + \pi
\end{eqnarray*}
which gives the formula for $k = 1$, then we conclude easily.
\end{proof}

\begin{prop}
The function $g$ satisfies

$$(g^{2} - \overline{g}^{2})g_{z \overline{z}} = 2g g_{z} g_{\overline{z}}$$
and its Hopf differential is

\begin{eqnarray}
Q = \dfrac{iK}{8} dz^{2}. \label{eq9}
\end{eqnarray}
\end{prop}

\begin{proof}[Proof]
A direct calculation shows that $g$ satisfies the equation. Hence, the Hopf differential is given by

\begin{eqnarray*}
 Q & = & \dfrac{g_{z} \overline{g}_{z}}{g^{2} - \overline{g}^{2}} dz^{2} \\
           & = & \dfrac{i(1-b'^{2})}{8 \cos{(2b)}} dz^{2} \\
           & = & \dfrac{iK}{8} dz^{2}.
\end{eqnarray*}
\end{proof}

Thus the map $g$ induces a conformal minimal immersion $x = (x_{1},x_{2},x_{3})$ such that

\begin{eqnarray*}
 {x_{1}}_{z} & = & e^{-x_{3}} \dfrac{(\overline{g}^{2} - 1)g_{z}}{g^{2} - \overline{g}^{2}} = \dfrac{[1+ie^{-2u} e^{-2ib} ](1 - b')e^{ib}e^{i\pi/4}}{4e^{-u} \cos{(2b)}} e^{-x_{3}} \\
 {x_{2}}_{z} & = & ie^{x_{3}} \dfrac{(\overline{g}^{2} + 1)g_{z}}{g^{2} - \overline{g}^{2}} = -\dfrac{[1-ie^{-2u} e^{-2ib} ]i(1 - b')e^{ib}e^{i\pi/4}}{4e^{-u} \cos{(2b)}} e^{x_{3}} \\
 {x_{3}}_{z} & = & \dfrac{2\overline{g}g_{z}}{g^{2} - \overline{g}^{2}} = \dfrac{i(b'-1)}{2 \cos{(2b)}}.
\end{eqnarray*}
This application is an immersion since the metric induced by $x$ is given by

\begin{eqnarray*}
dw^{2} & = & \left\| x_{u} \right\|^{2} |dz|^{2} \\
       & = & \dfrac{K^{2}}{(1+b')^{2}} \cosh^{2}{(u)} |dz|^{2}.
\end{eqnarray*}
We obtain immediately that $x_{3}$ is a one-variable function and satisfies

$$x_{3}'(v) = \dfrac{1-b'(v)}{\cos{(2b(v))}} = \dfrac{K}{1+b'(v)}.$$

\begin{rem}
 For $K=0$, we get $x_{3} =$ cst and the image of $x$ is a point. We will always exclude this case in the sequel.
\end{rem}

By setting $x_{3}(0) = 0$, we choose $x_{3}$ among the primitive functions.

\begin{prop}
\begin{enumerate}
\item The function $x_{3}$ is defined on the whole $\mathbb{R}$ and is bijective ;
\item The function $x_{3}$ is odd ;
\item The function $x_{3}$ satisfies

$$x_{3}(v+W) = x_{3}(v) + x_{3}(W)$$
for all real number $v$.
\end{enumerate}
\end{prop}

\begin{proof}[Proof]
\begin{enumerate}
\item Obvious because $x_{3}$ is a primitive of a continuous function, and its derivative has the sign of $K$.
\item The function $b$ is odd then $b'$ is even so $x_{3}'$ is even and $x_{3}$ is odd.
\item We have $x_{3}'(v+W) = x_{3}'(v)$ and the result follows.
\end{enumerate}
\end{proof}

Hence, the calculus proves that the functions

$$\begin{array}{rcl}
   x_{1}(u+iv) & = & \dfrac{\sqrt{2}}{2} (\cos{(b(v))} - \sin{(b(v))}) x_{3}' e^{-x_{3}} \sinh{(u)} \\
   x_{2}(u+iv) & = & \dfrac{\sqrt{2}}{2} (\cos{(b(v))} + \sin{(b(v))}) x_{3}' e^{x_{3}} \sinh{(u)}.
  \end{array}$$
satisfy the equations above.

\begin{theo}
\label{theohelicoide}
 Let $K$ be a real number such that $|K| < 1$ and $K \neq 0$, and $b$ the function defined by \eqref{eq8}. We define the function $x_{3}$ by
 
$$x_{3}' = \dfrac{K}{1+b'},~x_{3}(0)=0,$$
Then the map

$$x : u+iv \in \mathbb{C} \longmapsto \begin{pmatrix}
       \dfrac{\sqrt{2}}{2} (\cos{(b(v))} - \sin{(b(v))}) x_{3}' e^{-x_{3}} \sinh{(u)} \\
       \dfrac{\sqrt{2}}{2} (\cos{(b(v))} + \sin{(b(v))}) x_{3}' e^{x_{3}} \sinh{(u)} \\
       x_{3}(v)
      \end{pmatrix}$$
is a conformal minimal immersion whose Gauss map is $g : u+iv \in \mathbb{C} \longmapsto e^{-u} e^{ib(v)} e^{-i\pi/4}$. Moreover,

\begin{eqnarray}
(0,0,2x_{3}(W)) \ast x(u+iv) = x(u+i(v+2W)) \label{eq10}
\end{eqnarray}
for all $u, v \in \mathbb{R}$. The surface given by $x$ is called a \emph{helicoid} of parameter $K$ and will be denoted by $\mathcal{H}_{K}$.
\end{theo}

\begin{proof}[Proof]
The equation \eqref{eq10} means that the helicoid is invariant by left multiplication by $(0,0,2x_{3}(W))$. Recall that we have the identity

$$x_{3}(v+2W) = x_{3}(v+W) + x_{3}(W) = x_{3}(v) + 2x_{3}(W)$$
for all real number $v$. Thus we get the result for the third coordinate and we prove in the same way that $e^{-2x_{3}(W)}x_{1}(u+iv) = x_{1}(u+i(v+2W))$ and $e^{2x_{3}(W)}x_{2}(u+iv) = x_{2}(u+i(v+2W))$.
\end{proof}

\begin{rem}
\begin{enumerate}
\item The surface $\mathcal{H}_{K}$ is embedded because $x_{3}$ is bijective. It is easy to see that it is even properly embedded.
\item The surfaces $\mathcal{H}_{K}$ and $\mathcal{H}_{-K}$ are related ; if we denote by an index $K$ (resp. $-K$) the datas describing $\mathcal{H}_{K}$ (resp. $\mathcal{H}_{-K}$), we get

$$\left\{ \begin{array}{rcl}
b_{-K}(v) & = & b_{K}(v+W/2) - \pi / 2 \\
{x_{3}}_{-K}(v) & = & -{x_{3}}_{K}(v+W/2) + {x_{3}}_{K}(W/2).
\end{array} \right.$$
In particular, ${x_{3}}_{-K}(W) = -{x_{3}}_{K}(W)$ and both surfaces have the same period $|{x_{3}}_{K}(W)|$. Finally,

$$x_{-K}(u+iv) = (0,0,{x_{3}}_{K}(W/2)) \ast \sigma^{3} \circ x_{K}(u+i(v+W/2)).$$
Thus, there exists an isometry of Sol$_{3}$ who puts $\mathcal{H}_{-K}$ on $\mathcal{H}_{K}$.
\end{enumerate}
\end{rem}

\begin{prop}
For every real number $T$, there exists a unique helicoid $\mathcal{H}_{K}$ (up to isometry, i.e. up to $K \leftmapsto -K$) whose period is $T$.
\end{prop}

\begin{proof}
We did notice that the period of the helicoid $\mathcal{H}_{K}$ is

\begin{eqnarray*}
2x_{3}(W) := 2{x_{3}}_{K}(W) & = & 2 \int_{0}^{W} \dfrac{K}{1+b'(s)} ds \\
		  & = & 2K \int_{0}^{\pi} \dfrac{du}{\sqrt{1-K\cos{(2u)}} (1+ \sqrt{1-K\cos{(2u)}})}
\end{eqnarray*}
with the change of variables $u = b(s)$ and $b(W) = \pi$. Seeing ${x_{3}}_{K}(W)$ as a function of the variable $K$, we get

\begin{eqnarray*}
\dfrac{\partial {x_{3}}_{K}(W)}{\partial K} & = &  \int_{0}^{\pi} \dfrac{1}{\sqrt{1-K \cos{(2u)}}^{3}} du
\end{eqnarray*}
(valid for $K$ in every compact set $[0,a] \subset [0,1[$, and so in $[0,1[$), and so the function $K \longmapsto {x_{3}}_{K}(W)$ is injective. Moreover, we have ${x_{3}}_{0}(W) = 0$ and

\begin{eqnarray*}
{x_{3}}_{1}(W) & = & \int_{0}^{\pi} \dfrac{1}{\sqrt{1-\cos{(2u)}}(1+\sqrt{1-\cos{(2u)}})} du \\
               & = & \int_{0}^{\pi} \dfrac{1}{\sqrt{2} \sin{(u)}(1+\sqrt{2}\sin{(u)})} du \\ 				   & = & \dfrac{1}{\sqrt{2}} \int_{0}^{\infty} \dfrac{1+t^{2}}{1+2\sqrt{2}t+t^{2}} dt = + \infty, 
\end{eqnarray*}
and so ${x_{3}}_{K}(W)$ is a bijection from $]0,1[$ onto $]0,+\infty[$.
\end{proof}

The vector field defined by

\begin{eqnarray*}
N & = & \dfrac{1}{1 + \left| g \right|^{2}} \left[ \begin{matrix} 2 \Re{(g)} \\ 2 \Im{(g)} \\ 1 - \left| g \right|^{2} \end{matrix} \right] \\
  & = &  \dfrac{\sqrt{2}}{2 \cosh{(u)}} \left[ \begin{matrix} \cos{(b)} + \sin{(b)} \\ \sin{(b)} - \cos{(b)} \\ \sqrt{2} \sinh{(u)} \end{matrix} \right].
\end{eqnarray*}
is normal to the surface. We get

\begin{eqnarray*}
\nabla_{x_{u}} N & = & - \sin{(2b)} \dfrac{\sinh{(u)}}{\cosh{(u)}} x_{u} + \left( \dfrac{1+b'}{K \cosh^{2}{(u)}} - \cos{(2b)} \right) x_{v},\\
\nabla_{x_{v}} N & = & \left( \dfrac{1+b'}{K \cosh^{2}{(u)}} - \cos{(2b)} \right) x_{u} + \sin{(2b)} \dfrac{\sinh{(u)}}{\cosh{(u)}} x_{v} ,
\end{eqnarray*}
and thus the Gauss curvature is given by

$$\mathcal{K} = -1 + \dfrac{1}{\cosh^{2}{(u)}} \left( \dfrac{2(1+b')\cos{(2b)}}{K} - \dfrac{(1+b')^{2}}{K^{2} \cosh^{2}{(u)}} + \sin^{2}{(2b)} \right).$$
In particular, the fundamental pieces of the helicoids have infinite total curvature since

$$\mathcal{K} dA = \left( -\dfrac{K^{2}}{(1+b')^{2}} \cosh^{2}{(u)} + \dfrac{2K\cos{(2b)}}{1+b'} - \dfrac{1}{\cosh^{2}{(u)}} + \dfrac{K^{2} \sin^{2}{(2b)}}{(1+b')^{2}} \right) dudv.$$

\vspace{1\baselineskip}

We notice that

\begin{eqnarray*}
x(-u+iv) & = & \begin{pmatrix} -x_{1}(u+iv) \\ -x_{2}(u+iv) \\ x_{3}(v) \end{pmatrix} \\
			& = & \sigma^{2} \circ x(u+iv),
\end{eqnarray*}
where $\sigma$ and $\tau$ are the isometries introduced in the first section : the helicoid $\mathcal{H}_{K}$ is symmetric by rotation of angle $\pi$ around the $x_{3}$-axis, which is included in the helicoid as the image by $x$ of the pure imaginary axis of $\mathbb{C}$. On this axis we have

$$g(0+iv) = -ie^{ib(v)}.$$
Hence, the straight line $\{(x,x,0)~|~x \in \mathbb{R} \}$ is included in  the helicoid as the image by $x$ of the real line. Along this line, we have

$$g(u+i0) = e^{-u}e^{-i \pi / 4}.$$
Then we notice that

\begin{eqnarray*}
x(u-iv) & = & \begin{pmatrix} x_{2}(u+iv) \\ x_{1}(u+iv) \\ -x_{3}(v) \end{pmatrix} \\
			& = & \sigma \tau \circ x(u+iv),
\end{eqnarray*}
Thus, $\mathcal{H}_{K}$ is symmetric by rotation of angle $\pi$ around the axis $\{(x,x,0)~|~x \in \mathbb{R} \}$.

\begin{rem}
The straight line $\{(x,x,0)~|~x \in \mathbb{R} \}$ is a geodesic of the helicoid. It's even a geodesic of Sol$_{3}$.
\end{rem}

Since the function $\sinh$ is odd, we deduce that

\begin{eqnarray*}
x(-u-iv) & = & \begin{pmatrix} -x_{2}(u+iv) \\ -x_{1}(u+iv) \\ -x_{3}(v) \end{pmatrix} \\
			& = & \sigma^{3} \tau \circ x(u+iv).
\end{eqnarray*}
Thus, $\mathcal{H}_{K}$ is symmetric by rotation of angle $\pi$ around the axis $\{(x,-x,0)~|~x \in \mathbb{R} \}$ (but this axis is not included in the surface).

\vspace{1\baselineskip}

The helicoid $\mathcal{H}_{K}$ has no more symmetry fixing the origin ; indeed if it had, it would exist a diffeomorphism $\phi$ of $\mathbb{C}$ such that $x \circ \phi = \sigma^{2} \circ x$ (I choose $\sigma^{2}$ as an example but it is the same idea for the other elements of the isotropy group of the origin of Sol$_{3}$). By composition, the surface would have the whole symmetries of the isotropy group. But if $x \circ \phi = \tau \circ x$, the decomposition $\phi = \phi_{1} + i\phi_{2}$ leads to

\begin{eqnarray*}
\begin{pmatrix} x_{1}(\phi_{1}(u+iv)+i\phi_{2}(u+iv)) \\ x_{2}(\phi_{1}(u+iv)+i\phi_{2}(u+iv)) \\ x_{3}(\phi_{2}(u+iv)) \end{pmatrix} & = & \begin{pmatrix} -x_{1}(u+iv) \\ x_{2}(u+iv) \\ x_{3}(v) \end{pmatrix}.
\end{eqnarray*}
Because $x_{3}$ is bijective, we get, $\phi_{2}(u+iv) = v$ for all $u, v$, and then we get at the same time $\sinh{(\phi_{1}(u+iv))} = \sinh{(u)}$ and $\sinh{(\phi_{1}(u+iv))} = -\sinh{(u)}$, what is impossible.

\begin{figure}[H]
    \centering
        \includegraphics[height=10cm,width=15cm]{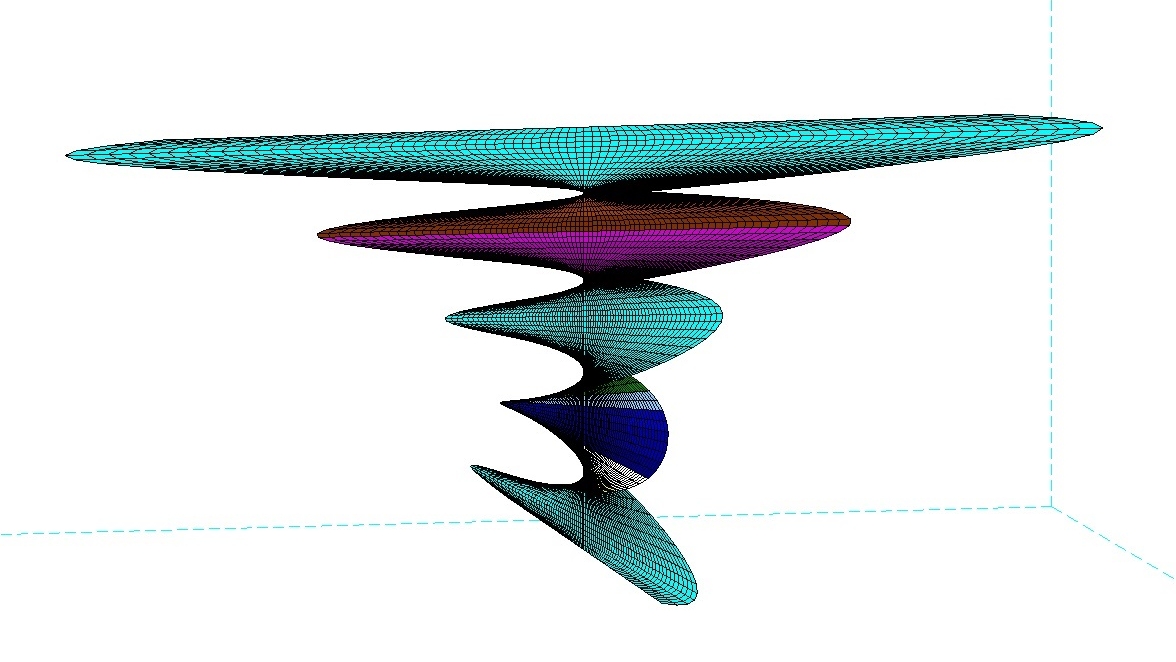}
    \caption{Helicoid for $K = 0.4$ with Scilab®}
\end{figure}

%% file: Partie4.tex
\section{Catenoids in Sol$_{3}$}

In this section we construct examples of minimal annuli in Sol$_{3}$. Let $\alpha \in~]-1;1[$. We start from a map $g$ defined on $\mathbb{C}$ by

$$g(z=u+iv) = -ie^{-u-\gamma(v)} e^{i\rho(v)},$$
where $\rho$ satisfies the following ODE :

\begin{eqnarray}
\rho' & = & \sqrt{1 - \alpha^{2} \sin^{2}{(2\rho)}},~~\rho(0) = 0~ \label{eq11}
\end{eqnarray}
and $\gamma$ is defined by

\begin{eqnarray}
\gamma' & = & -\alpha \sin{(2 \rho)},~~\gamma(0) = 0. \label{eq12}
\end{eqnarray}

\begin{prop}
\label{propperiodecat}
The map $\rho$ is well defined and has the following properties :
\begin{enumerate}
\item The function $\rho$ is an increasing diffeomorphism from $\mathbb{R}$ onto $\mathbb{R}$ ;
\item The function $\rho$ is odd ;
\item There exists a real number $V > 0$ such that

$$\forall v \in \mathbb{R},~~~\rho(v+V) = \rho(v) + \pi ;$$
\item The function $\rho$ satisfies $\rho(kV) = k \pi$ for all $k \in \mathbb{Z}$.
\end{enumerate}
\end{prop}

\begin{proof}[Proof]
Since $\alpha \in~]-1,1[$, there exists $r > 0$ such that $1 - \alpha^{2} \sin^{2}{(2\rho)} \in~]r,1]$ ; the Cauchy-Lipschitz theorem can be applied, and $\rho$ is well defined.
\begin{enumerate}
\item By \eqref{eq11} we have $\rho' > 0$ on its definition domain, and $\sqrt{r} < \rho' < 1$. So $\rho'$ is bounded by two positive constants, then $\rho$ is defined on the entire $\mathbb{R}$, and

$$\lim\limits_{v \to \pm \infty} \rho(v) = \pm \infty.$$

\item The function $\widehat{\rho} : v \longmapsto -\rho(-v)$ satisfies \eqref{eq11} with $\widehat{\rho}(0) = 0$ ; hence $\widehat{\rho} = \rho$ and $\rho$ is odd.

\item There exists $V > 0$ such that $\rho(V) = \pi$ ; Then the function $\widetilde{\rho} : v \longmapsto \rho(v+V) - \pi$ satisfies \eqref{eq11} with $\widetilde{\rho}(0) = 0$ ; hence $\widetilde{\rho} = \rho$.

\item Obvious.
\end{enumerate}
\end{proof}

\begin{cor}
\label{corollairerho}
\begin{enumerate}
\item We have $\rho(kV/2) = k \pi / 2$ for all $k \in 2 \mathbb{Z} +1$ ;
\item We have

$$\rho \left( -v + \dfrac{V}{2} \right) = -\rho(v) + \dfrac{\pi}{2}$$
for all $v \in \mathbb{R}$. In particular,

$$\rho \left( \dfrac{V}{4} \right) = \dfrac{\pi}{4},~~~\rho \left( \dfrac{3V}{4} \right) = \dfrac{3 \pi}{4}.$$
\end{enumerate}
\end{cor}

\begin{proof}[Proof]

\begin{enumerate}
\item We have
\begin{eqnarray*}
\rho \left( \dfrac{V}{2} \right) & = & \rho \left( -\dfrac{V}{2} + V \right) \\
		& = & - \rho \left( \dfrac{V}{2} \right) + \pi
\end{eqnarray*}
which gives the formula for $k = 1$, then we conclude easily.

\item The functions $\rho^{*} : v \longmapsto \pi / 2 - \rho(-v + V / 2)$ and $\rho$ satisfy equation \eqref{eq11} with $\rho^{*}(0) = \rho(0) = 0$, so $\rho^{*} = \rho$ and

\begin{eqnarray*}
\rho(V / 4) & = & \rho^{*}(V / 4) \\
            & = & \dfrac{\pi}{2} - \rho \left( \dfrac{\pi}{2} - \dfrac{\pi}{4} \right)
\end{eqnarray*}
and the result follows.
\end{enumerate}
\end{proof}

\begin{prop}
The function $g$ satisfies

$$(g^{2} - \overline{g}^{2})g_{z \overline{z}} = 2g g_{z} g_{\overline{z}}$$
and its Hopf differential is

\begin{eqnarray}
Q = -\dfrac{\alpha}{4} dz^{2}. \label{eq13}
\end{eqnarray}
\end{prop}

\begin{proof}[Proof]
A direct calculation shows that $g$ satisfies the equation. Hence, the Hopf differential is given by

\begin{eqnarray*}
 \widetilde{Q} & = & \dfrac{i(1-\rho'^{2}-\gamma'^{2} - 2i \gamma')}{8 \sin{(2\rho)}} dz^{2} \nonumber \\
	       & = & -\dfrac{\alpha}{4} dz^{2}.
\end{eqnarray*}
\end{proof}

Thus the map $g$ induces a conformal minimal immersion $x = (x_{1},x_{2},x_{3})$ such that

\begin{eqnarray*}
 {x_{1}}_{z} = e^{-x_{3}} \dfrac{(\overline{g}^{2} - 1)g_{z}}{g^{2} - \overline{g}^{2}},~~~~{x_{2}}_{z} = ie^{x_{3}} \dfrac{(\overline{g}^{2} + 1)g_{z}}{g^{2} - \overline{g}^{2}},~~~~{x_{3}}_{z} = \dfrac{2\overline{g}g_{z}}{g^{2} - \overline{g}^{2}}.
\end{eqnarray*}
This application is an immersion since the metric induced by $x$ is given by

\begin{eqnarray*}
dw^{2} & = & \left\| x_{u} \right\|^{2} |dz|^{2} \\
       & = & ({F'}^{2} + \alpha^{2}) \cosh^{2}{(u + \gamma)} |dz|^{2} \\
       & = & \left( \dfrac{\alpha^{4} \sin^{2}{(2 \rho)}}{(1+\rho')^{2}} + \alpha^{2} \right) \cosh^{2}{(u + \gamma)} |dz|^{2} \\
       & = & \dfrac{2 \alpha^{2}}{1+\rho'} \cosh^{2}{(u + \gamma)} |dz|^{2},
\end{eqnarray*}

In particular,

$${x_{3}}_{z} = \dfrac{i\rho'-\gamma'-i}{2 \sin{(2\rho)}},$$
that is

$$\left\{ \begin{array}{rcl}
 {x_{3}}_{u} & = & 2 \Re{({x_{3}}_{z})} = -\dfrac{\gamma'}{\sin{(2\rho)}} = \alpha \\
 {x_{3}}_{v} & = & -2\Im{({x_{3}}_{z})} = \dfrac{1-\rho'}{\sin{(2\rho)}} = \dfrac{\alpha^{2}\sin{(2\rho)}}{1+\rho'}.
\end{array} \right.$$
Thus

\begin{eqnarray*}
 x_{3}(u+iv) & = & \alpha u + \alpha^{2} \int^{v} \dfrac{\sin{(2\rho(s))}}{1+\rho'(s)} ds.
\end{eqnarray*} 
Here we have to choose an initial condition ; we set
 
$$F(v) = \alpha^{2} \int_{0}^{v} \dfrac{\sin{(2\rho(s))}}{1+\rho'(s)} ds$$
and define

\begin{eqnarray*}
 x_{3}(u+iv) & = & \alpha u + F(v).
\end{eqnarray*}
The function $F$ is well defined on the whole $\mathbb{R}$.

\begin{prop}
\label{propperiodeF}
The function $F$ is even and $V$-periodic.
\end{prop}

\begin{proof}[Proof]
The function $F'$ is odd because $\rho$ is odd and $\rho'$ is even, so $F$ is even. Then we get

\begin{eqnarray*}
F'(v+V) & = & \alpha^{2} \dfrac{\sin{(2\rho(v)+2\pi)}}{1+\rho'(v)} \\
        & = & F'(v),
\end{eqnarray*}
so there exists a constant $C$ such that

$$F(v+V) = F(v) + C$$
for all $v \in \mathbb{R}$. By evaluating in zero, we get $C = F(V)$, i.e.

\begin{eqnarray*}
C & = & \alpha^{2} \int_{0}^{V} \underbrace{\dfrac{\sin{(2\rho(s))}}{1+\rho'(s)}}_{:= H(s)} ds \\
  & = & \alpha^{2} \left\{ \int_{0}^{V/4} H(s) ds + \int_{V/4}^{3V/4} H(s) ds + \int_{3V/4}^{V} H(s) ds \right\} \\
	    & := & \sum_{k=0}^{2} L_{k}(\alpha).
\end{eqnarray*}
We can now do the change of variables $u = \sin{(2 \rho(s))}$ in each integral $L_{k}(\alpha)$, with

$$du = 2 \rho'(s) \cos{(2 \rho(s))} ds = 2 (-1)^{k} \sqrt{(1-\alpha^{2}u^{2}) (1-u^{2})} ds.$$
Thus,

\begin{eqnarray*}
C & = & \alpha^{2} \int_{-1}^{1} \dfrac{u~du}{(1+\sqrt{1-\alpha^{2} u^{2}}) \sqrt{(1-\alpha^{2} u^{2}) (1-u^{2})}} \\
  & = & 0
\end{eqnarray*}
and $F$ is $V$-periodic.
\end{proof}

\begin{prop}
\label{periodgamma}
The function $\gamma$ is even and $V$-periodic.
\end{prop}

\begin{proof}[Proof]
We prove the proposition in exactly the same way than the function $F$.
\end{proof}

The two other equations becomes

\begin{eqnarray*}
 {x_{1}}_{z} & = & e^{-x_{3}} \dfrac{(e^{-u-\gamma-i\rho} + e^{u+\gamma+i\rho})(1-\rho'-i\gamma')}{4 \sin{(2\rho)}} \\
 {x_{2}}_{z} & = & -e^{x_{3}} \dfrac{(e^{u+\gamma+i\rho} - e^{-u-\gamma-i\rho})(i-i\rho'+\gamma')}{4 \sin{(2\rho)}}.
\end{eqnarray*}
Those equations lead to

\begin{eqnarray*}
x_{1} & = & e^{-\alpha u-F} \Big[ \dfrac{e^{u + \gamma}}{2(1-\alpha)}(\cos{(\rho)}F'- \alpha \sin{(\rho)}) - \dfrac{e^{-u - \gamma}}{2(1+\alpha)}(\alpha \sin{(\rho) + \cos{(\rho)}F'}) \Big]~; \\
x_{2} & = & e^{\alpha u+F} \Big[ \dfrac{-e^{u + \gamma}}{2(1+\alpha)}(\alpha \cos{(\rho)} + F'\sin{(\rho)}) + \dfrac{e^{-u - \gamma}}{2(\alpha-1)}(\alpha \cos{(\rho)} - F' \sin{(\rho)}) \Big].
\end{eqnarray*}

\begin{rem}
If $\alpha = 0$, $x(\mathbb{C}) = \{ 0 \}$. This case will be excluded in the sequel.
\end{rem}

\begin{theo}
\label{theocatenoide}
 Let $\alpha$ be a real number such that $|\alpha| < 1$ and $\alpha \neq 0$, and $\rho$ and $\gamma$ the functions defined by \eqref{eq11} and \eqref{eq12}. We define the function $F$ by
 
$$F(v) = \alpha^{2} \int_{0}^{v} \dfrac{\sin{(2\rho(s))}}{1+\rho'(s)} ds.$$
Then the map $x : \mathbb{C} \longrightarrow \emph {Sol}_{3}$ defined by

$$\begin{pmatrix}
 e^{-\alpha u-F} \Big[ \dfrac{e^{u + \gamma}}{2(1-\alpha)}(\cos{(\rho)}F'-\alpha \sin{(\rho)}) - \dfrac{e^{-u - \gamma}}{2(1+\alpha)}(\alpha \sin{(\rho) + \cos{(\rho)}F'}) \Big] \\
 e^{\alpha u+F} \Big[ \dfrac{-e^{u + \gamma}}{2(1+\alpha)}(\alpha \cos{(\rho)} + F'\sin{(\rho)}) + \dfrac{e^{-u - \gamma}}{2(\alpha-1)}(\alpha \cos{(\rho)} - F' \sin{(\rho)}) \Big] \\
 \alpha u + F
\end{pmatrix}$$
is a conformal minimal immersion whose Gauss map is $g : u+iv \in \mathbb{C} \longmapsto -ie^{-u-\gamma(v)} e^{i\rho(v)}$. Moreover,

\begin{eqnarray}
x(u+i(v+2V)) = x(u+iv) \label{eq14}
\end{eqnarray}
for all $u, v \in \mathbb{R}$. The surface given by $x$ is called a \emph{catenoid} of parameter $\alpha$ and will be denoted by $\mathcal{C}_{\alpha}$.
\end{theo}

\begin{proof}[Proof]
The periodicity of $\mathcal{C}_{\alpha}$ is an application of propositions \ref{propperiodecat}, \ref{propperiodeF} and \ref{periodgamma}.
\end{proof}

\begin{rem}
The surfaces $\mathcal{C}_{\alpha}$ and $\mathcal{C}_{-\alpha}$ are related ;  if we denote by an index $\alpha$ (resp. $-\alpha$) the datas describing $\mathcal{C}_{\alpha}$ (resp. $\mathcal{C}_{-\alpha}$), we get

$$\left\{ \begin{array}{rcl}
\rho_{-\alpha} & = & \rho_{\alpha} \\
F_{-\alpha} & = & F_{\alpha} \\
\gamma_{-\alpha} & = & -\gamma_{\alpha} \\
\end{array} \right.$$
Thus, we get

$$x_{-\alpha}(-u+iv) = \sigma^{2} \circ x_{\alpha}(u+iv).$$
In particular, there exists an orientation-preserving isometry of Sol$_{3}$ fixing the origin who puts $\mathcal{C}_{\alpha}$ on $\mathcal{C}_{-\alpha}$.

\end{rem}

Now we show that the catenoids are embedded :

\begin{prop}
\label{courbeconvexe}
For all $\lambda \in \mathbb{R}$, the intersection of $\mathcal{C}_{\alpha}$ with the plane $\{ x_{3} = \lambda \}$ is a non-empty closed embedded convex curve.
\end{prop}

\begin{proof}[Proof]
This intersection is non-empty : $x(\lambda / \alpha + i0) \in \mathcal{C}_{\alpha} \cap \{ x_{3} = \lambda \}$. We look at the curve in $\mathbb{C}$ defined by $x_{3}(u+iv) = \alpha u + F(v) = \lambda$, i.e. the curve

$$c : t \in \mathbb{R} \longmapsto  \Big( \dfrac{\lambda - F(t)}{\alpha},t \Big).$$
Its image by $x$ is

$$\mathfrak{c} : t \in \mathbb{R} \longmapsto \begin{pmatrix} e^{-\lambda} \Big[ \dfrac{e^{\frac{\lambda - F}{\alpha} + \gamma}}{2(1-\alpha)}(\cos{(\rho)}F'-\alpha \sin{(\rho)}) - \dfrac{e^{-\frac{\lambda - F}{\alpha} - \gamma}}{2(1+\alpha)}(\alpha \sin{(\rho) + \cos{(\rho)}F'}) \Big] \\
 e^{\lambda} \Big[ \dfrac{-e^{\frac{\lambda - F}{\alpha} + \gamma}}{2(1+\alpha)}(\alpha \cos{(\rho)} + F'\sin{(\rho)}) + \dfrac{e^{-\frac{\lambda - F}{\alpha} - \gamma}}{2(\alpha-1)}(\alpha \cos{(\rho)} - F' \sin{(\rho)}) \Big] \\
 c \end{pmatrix}(t)$$
The calculation leads to

$$\mathfrak{c}_{1}'(t) = \dfrac{e^{-\lambda}}{\alpha (1-\alpha^{2})} \left[ A(t) \cosh{\left( \frac{\lambda - F}{\alpha} + \gamma \right)} + B(t) \sinh{\left( \frac{\lambda - F}{\alpha} + \gamma \right)} \right],$$
with

$$\left\{ \begin{array}{rcl}
 A & = & -F'^{2} \cos{(\rho)} + \alpha \gamma' F' \cos{(\rho)} - \alpha^{2} \rho' \cos{(\rho)} + \alpha^{2} F' \sin{(\rho)} - \alpha^{3} \gamma' \sin{(\rho)} + \alpha^{2} F'' \cos{(\rho)} \\
   &   & ~~~~~~~~~~ - \alpha^{2} F' \rho' \sin{(\rho)} \\
 B & = & \alpha F' \sin{(\rho)} - \alpha^{2} \gamma' \sin{(\rho)} + \alpha F'' \cos{(\rho)} - \alpha F' \rho' \sin{(\rho)} - \alpha F'^{2} \cos{(\rho)} \\
   &   & ~~~~~~~~~~ + \alpha^{2} \gamma' F' \cos{(\rho)} - \alpha^{3} \rho' \cos{(\rho)}
\end{array} \right.$$
We remark that $B \equiv 0$ after simplifications, and

$$A(t) = (F'^{2}(t) + \alpha^{2})(\alpha^{2}-1) \cos{(\rho(t))}.$$
Finally,

$$\mathfrak{c}_{1}'(t) = -\dfrac{e^{-\lambda}}{\alpha} (F'^{2}(t) + \alpha^{2}) \cos{(\rho(t))} \cosh{\left( \frac{\lambda - F(t)}{\alpha} + \gamma(t) \right)}.$$
In the same way, we get

$$\mathfrak{c}_{2}'(t) = -\dfrac{e^{-\lambda}}{\alpha} (F'^{2}(t) + \alpha^{2}) \sin{(\rho(t))} \cosh{\left( \frac{\lambda - F(t)}{\alpha} + \gamma(t) \right)}.$$
Thus

$$\mathfrak{c}_{1}'^{2} + \mathfrak{c}_{2}'^{2} = \dfrac{e^{-2\lambda}}{\alpha^{2}} (F'^{2}(t) + \alpha^{2})^{2} \cosh^{2}{\left( \frac{\lambda - F(t)}{\alpha} + \gamma(t) \right)} > 0.$$
so the intersection $\mathcal{C}_{\alpha} \cap \{ x_{3} = \lambda \}$ is a smooth curve ; moreover, it's closed since $\mathfrak{c}(t+2V) = \mathfrak{c}(t)$ for all $t \in \mathbb{R}$.

\vspace{1\baselineskip}

The planes $\{x_{3} = \lambda\}$ are flat : indeed, the metrics on these planes are $e^{2\lambda} dx_{1}^{2} + e^{-2\lambda} dx_{2}^{2}$, so up to an affine transformation, we can work in euclidean coordinates as we suppose to be in the proof since affinities preserve convexity.

To prove that $\mathfrak{c}$ is embedded and convex, we consider the part of $\mathfrak{c}$ corresponding to $t \in (-V/2 , V/2)$. On $(-V/2 , V/2)$, we have $\cos{(\rho(t))} > 0$ thanks to proposition \ref{propperiodecat} and corollary \ref{corollairerho}. So $\mathfrak{c}_{1}'(t) < 0$ if $\alpha > 0$ (and $\mathfrak{c}_{1}'(t) > 0$ if $\alpha < 0$) and $\mathfrak{c}_{1}$ is injective and decreasing if $\alpha > 0$ (and increasing if $\alpha < 0$). We get

$$\dfrac{d \mathfrak{c}_{2}}{d \mathfrak{c}_{1}} = \tan{(\rho(t))},$$
so $\frac{d \mathfrak{c}_{2}}{d \mathfrak{c}_{1}}$ is an increasing function of $t$, and also of $\mathfrak{c}_{1}$ if $\alpha < 0$ (and a decreasing function of the decreasing function $\mathfrak{c}_{1}$ if $\alpha > 0$). In both cases, the curve is convex.

Then, the half of $\mathfrak{c}$ corresponding to $t \in (-V/2,V/2)$ is convex and embedded. Since $\mathfrak{c}(t+V) = -\mathfrak{c}(t)$, the entire curve is convex and embedded.

\end{proof}

The following graphics show sections of the catenoid $\alpha = -0.6$ with planes $\{ x_{3} = \mbox{cte} \}$ :

\begin{figure}[H]
    \centering
        \includegraphics[height=6cm,width=9cm]{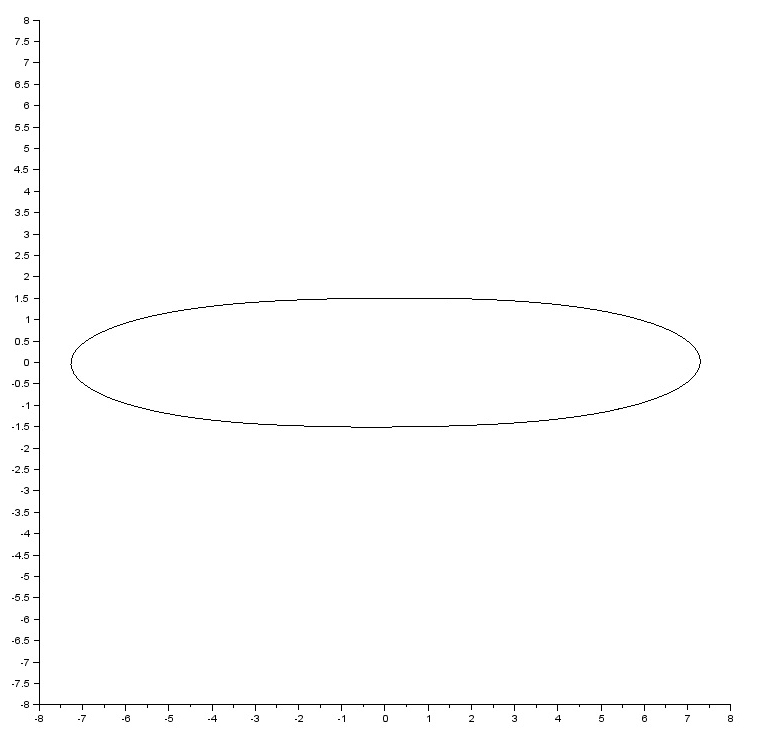}
    \caption{Section with $\{ x_{3} = -1 \}$ with Scilab®}
\end{figure}

\begin{figure}[H]
    \centering
        \includegraphics[height=6cm,width=9cm]{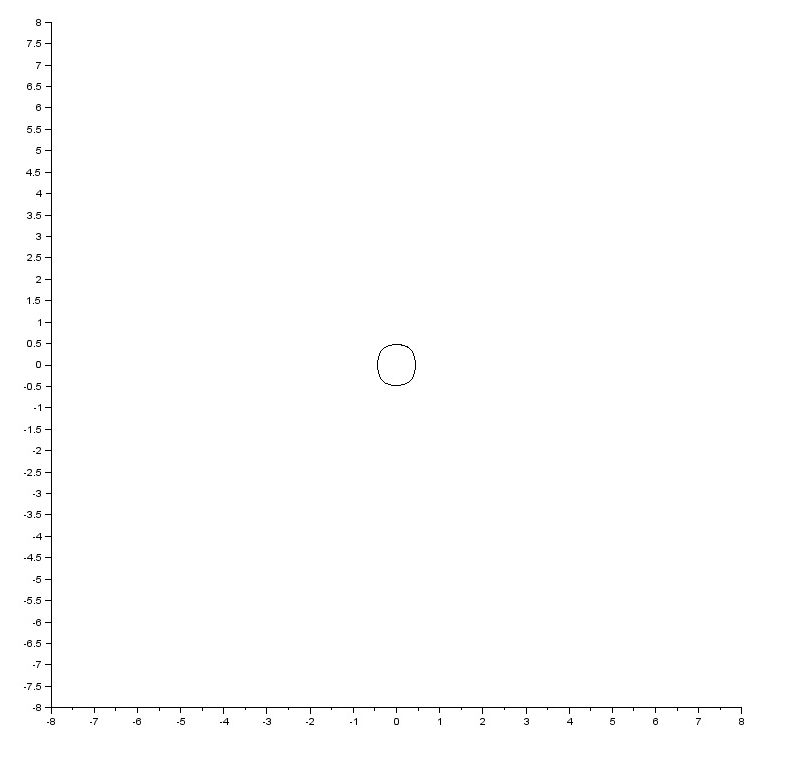}
    \caption{Section with $\{ x_{3} = 0 \}$ with Scilab®}
\end{figure}

\begin{figure}[H]
    \centering
        \includegraphics[height=6cm,width=9cm]{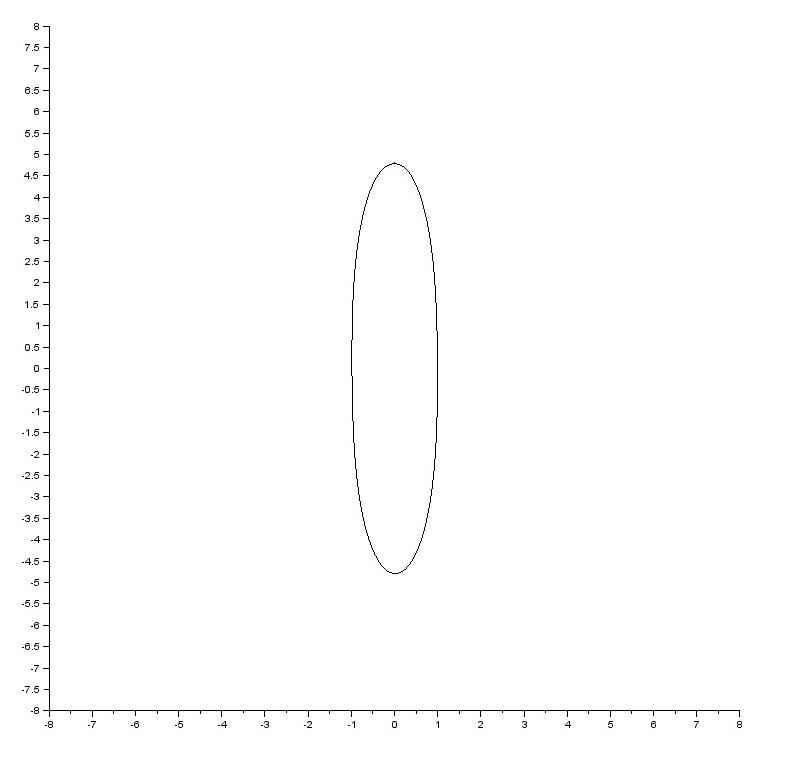}
    \caption{Section with $\{ x_{3} = 1 \}$ with Scilab®}
\end{figure}

\begin{cor}
 The surface $\mathcal{C}_{\alpha}$ is properly embedded for all $\alpha \in~]-1,1[~\backslash \{ 0 \}$.
\end{cor}

\begin{prop}
For all $\alpha \in~]-1,1[~\backslash \{ 0 \}$, the surface $\mathcal{C}_{\alpha}$ is conformally equivalent to $\mathbb{C}~\backslash \{ 0 \}$.
\end{prop}

\begin{proof}[Proof]
The map $x : \mathbb{C} / (2iV \mathbb{Z}) \longrightarrow \mathcal{C}_{\alpha}$ is a conformal bijective parametrization of $\mathcal{C}_{\alpha}$.
\end{proof}

The vector field defined by

\begin{eqnarray*}
N & = & \dfrac{1}{\cosh{(u)}} \left[ \begin{matrix} e^{-\gamma} \sin{(\rho)} \\ -e^{-\gamma} \cos{(\rho)} \\ \sinh{(u)} \end{matrix} \right].
\end{eqnarray*}
is normal to the surface. 

\vspace{1\baselineskip}

We have

$$\begin{array}{rcl}
x(u+i(v+V)) & = & \begin{pmatrix} -x_{1}(u+iv) \\ -x_{2}(u+iv) \\ x_{3}(u+iv) \end{pmatrix} \\
			& = & \sigma^{2} \circ x(u+iv).
\end{array}$$
Thus, the surface $\mathcal{C}_{\alpha}$ is symmetric by rotation of angle $\pi$ around the $x_{3}$-axis.

\begin{rem}
The $x_{3}$-axis is contained in the "interior" of $\mathcal{C}_{\alpha}$ since each curve $\mathcal{C}_{\alpha} \cap \{ x_{3} = \lambda \}$ is convex and symmetric with respect to this axis.
\end{rem}

We get as well

$$\begin{array}{rcl}
x(u-iv) & = & \begin{pmatrix} -x_{1}(u+iv) \\ x_{2}(u+iv) \\ x_{3}(u+iv) \end{pmatrix} \\
		& = & \tau \circ x(u+iv),
\end{array}$$
and the surface $\mathcal{C}_{\alpha}$ is symmetric by reflection with respect to the plane $\{ x_{1} = 0 \}$, and finally we have too

$$x(u+i(-v+V)) = \sigma^{2} \tau \circ x(u+iv)$$
and $\mathcal{C}_{\alpha}$ is symmetric by reflection with respect to the plane $\{ x_{2} = 0 \}$.

\vspace{1\baselineskip}

If $\mathcal{C}_{\alpha}$ had another symmetry fixing the origin, it would have every symmetry of the isotropy group of Sol$_{3}$, and we prove as for the helicoid that it is impossible.

\begin{figure}[H]
    \centering
        \includegraphics[height=10cm,width=15cm]{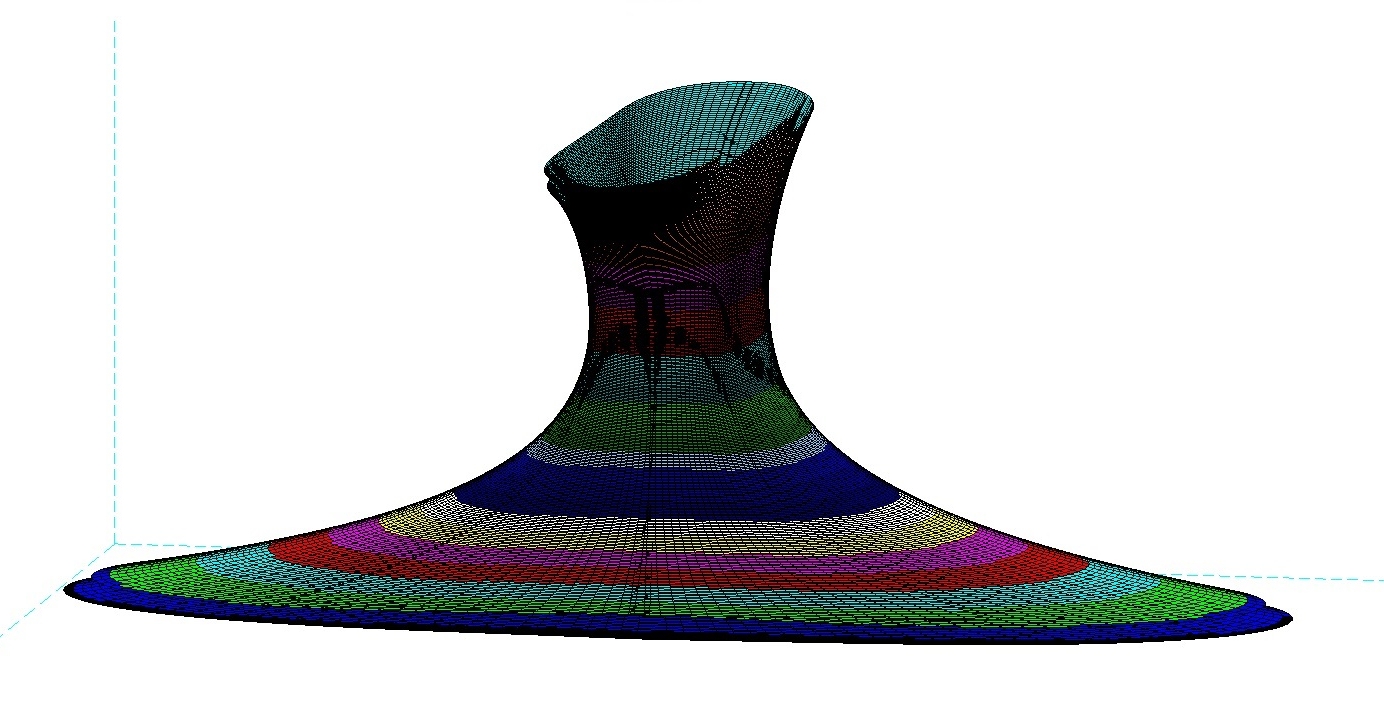}
    \caption{Catenoid for $\alpha = -0.6$ with Scilab®}
\end{figure}

%% file: Partie5.tex
\section{Limits of catenoids}

\subsection{The case $\alpha = 0$}

In this part we take a look to the limit surface of the catenoids $\mathcal{C}_{\alpha}$ when $\alpha$ goes to zero. For this, we do the change of parameters :

$$\left\{ \begin{array}{rcl}
u' & = & u + \ln{(\alpha)} \\
v' & = & v.
\end{array} \right.$$
In this coordinates, the immersion $x$ given in theorem \ref{theocatenoide} takes the form

$$\begin{pmatrix}
 e^{\alpha \ln{(\alpha)} - \alpha u'-F} \Big[ \dfrac{e^{u' + \gamma}}{2\alpha(1-\alpha)}(\cos{(\rho)}F'-\alpha \sin{(\rho)}) - \dfrac{\alpha e^{-u' - \gamma}}{2(1+\alpha)}(\alpha \sin{(\rho) + \cos{(\rho)}F'}) \Big] \\
 e^{- \alpha \ln{(\alpha)} + \alpha u'+F} \Big[ \dfrac{-e^{u' + \gamma}}{2\alpha(1+\alpha)}(\alpha \cos{(\rho)} + F'\sin{(\rho)}) + \dfrac{\alpha e^{-u' - \gamma}}{2(\alpha-1)}(\alpha \cos{(\rho)} - F' \sin{(\rho)}) \Big] \\
 - \alpha \ln{(\alpha)} + \alpha u' + F
\end{pmatrix}$$
Letting $\alpha$ go to zero, we get

$$\left\{ \begin{array}{rcl}
\rho & \longrightarrow & Id \\
F / \alpha & \longrightarrow & 0 \\
F' / \alpha & \longrightarrow & 0 \\
\gamma & \longrightarrow & 0
\end{array} \right.$$
and so the limit immersion is

$$\begin{pmatrix}
 -\dfrac{e^{u'}}{2} \sin{(v')} \\
 -\dfrac{e^{u'}}{2} \cos{(v')} \\
 0
\end{pmatrix}.$$
Thus, we obtain a parametrization of the plane $\{ x_{3} = 0 \}$, which is the limit of the family $(\mathcal{C}_{\alpha})$ when $\alpha \rightarrow 0$.

\subsection{The case $\alpha = 1$}

We close this paper by the study of the case $\alpha = 1$ (the case $\alpha = -1$ is exactly the same). We show that the limit surface is a minimal entire graph :

\begin{prop}
Let $\mathbf{x} : \mathbb{R}^{2} \longrightarrow \emph{Sol}_{3}$ be defined by

$$\mathbf{x}(u+iv) = \begin{pmatrix} \mathbf{x}_{1} \\ \mathbf{x}_{2} \\ \mathbf{x}_{3} \end{pmatrix} = \begin{pmatrix} - \dfrac{\tanh{(v)}}{2}(1 + e^{-2u}) \\
 \dfrac{e^{2u}}{4} - \dfrac{u}{2} - \dfrac{\cosh{(2v)}}{4} \\
 u + \ln{(\cosh{(v)})}
\end{pmatrix}.$$

Then $\mathbf{x}$ is a minimal immersion and there exists a $C^{\infty}$-function $f$ defined on the entire $\mathbb{R}^{2}$ such that the image of $x$ (called $\mathcal{S}$) is the $\mathbf{x}_{2}$-graph of $f$ given by $\mathbf{x}_{2} = f(\mathbf{x}_{1},\mathbf{x}_{3})$.
\end{prop}

\begin{proof}
We show that this surface is (up to a translation) the limit surface of the family $(\mathcal{C}_{\alpha})_{\alpha \in ]-1,1[}$ when $\alpha$ goes to $1$. For $\alpha = 1$, the Gauss map is still given by

$$g(z=u+iv) = -ie^{-u-\gamma(v)} e^{i\rho(v)},$$
but $\rho$ satisfies the following ODE :

\begin{eqnarray}
\rho' & = & \cos{(2\rho)},~~\rho(0) = 0~, \label{eq15}
\end{eqnarray}
and $\gamma$ is still defined by

\begin{eqnarray}
\gamma' & = & - \sin{(2 \rho)},~~\gamma(0) = 0. \label{eq16}
\end{eqnarray}
We have explicit expressions for these functions, which are given by

\begin{eqnarray*}
\rho(v) & = & \arctan{(e^{2v})} - \pi / 4 = \arctan{(\tanh{(v)})} \\
\gamma(v) & = & -\dfrac{1}{2} \ln{(\cosh{(2v)})}.
\end{eqnarray*}
Thus by setting

$$F(v) = \int_{0}^{v} \dfrac{\sin{(2\rho(s))}}{1+\cos{(2 \rho(s))}} ds,$$
we obtain 

$$F(v) = \ln{(\cosh{(v)})}.$$
Then, the immersion $x$ is given by

$$x = \begin{pmatrix}
 -\dfrac{e^{-2u}}{2} \tanh{(v)} + \dfrac{e^{-v}}{2\cosh{(v)}} \\
 \dfrac{e^{2u}}{4} - \dfrac{u}{2} - \dfrac{\cosh{(2v)}}{4} \\
 u + \ln{(\cosh{(v)})}
\end{pmatrix}.$$
A unit normal vector field is given by

\begin{eqnarray*}
N & = & \dfrac{1}{1+e^{-2u} \cosh{(2v)}} \left[ \begin{matrix}
 2\sinh{(v)} e^{-u} \\
 -2\cosh{(v)} e^{-u} \\
 1-e^{-2u} \cosh{(2v)}
\end{matrix} \right] \\
  & = & \dfrac{1}{\cosh^{2}{(v)}\cosh{(u)} - \sinh^{2}{(v)} \sinh{(u)}} \left[ \begin{matrix}
 \sinh{(v)} e^{-u} \\
 -\cosh{(v)} e^{-u} \\
 \cosh^{2}{(v)}\sinh{(u)} - \sinh^{2}{(v)} \cosh{(u)}
\end{matrix} \right].
\end{eqnarray*}
Thus, we get

$$g(u+iv) = -ie^{-u}(\cosh{(v)} + i\sinh{(v)})$$
who satisfies the harmonic equation \eqref{eq6}. The metric induced by this immersion on the surface is

$$ds^{2} = (e^{-4u} \tanh^{2}{(v)} + e^{2u} \sinh^{2}{(u)} + 1) |dz|^{2}.$$

This surface is symmetric by reflection with respect to the plane $\{ x_{1} = 1/2 \}$ since

$$x(u+iv) = \begin{pmatrix}
 \dfrac{1}{2} - \dfrac{\tanh{(v)}}{2}(1 + e^{-2u}) \\
 \dfrac{e^{2u}}{4} - \dfrac{u}{2} - \dfrac{\cosh{(2v)}}{4} \\
 u + \ln{(\cosh{(v)})}
\end{pmatrix} = \begin{pmatrix}
 \dfrac{1}{2} +\widetilde{x_{1}}(u,v) \\
 x_{2}(u,v) \\
 x_{3}(u,v) 
\end{pmatrix},$$
and so

$$x(u-iv) = \begin{pmatrix}
 \dfrac{1}{2} - \widetilde{x_{1}}(u,v) \\
 x_{2}(u,v) \\
 x_{3}(u,v) 
\end{pmatrix}.$$
This property is equivalent to the property that the translated surface $(-1/2,0,0) \ast x(u+iv)$ is symmetric with respect to $\{ x_{1} = 0 \}$. This translated surface is the image of the immersion $\mathbf{x}$ defined by

$$\mathbf{x}(u+iv) = (-1/2,0,0) \ast x(u+iv) = \begin{pmatrix}
 - \dfrac{\tanh{(v)}}{2}(1 + e^{-2u}) \\
 \dfrac{e^{2u}}{4} - \dfrac{u}{2} - \dfrac{\cosh{(2v)}}{4} \\
 u + \ln{(\cosh{(v)})}
\end{pmatrix}.$$
Then, this surface is analytic (like any minimal surface in Sol$_{3}$), so it is a local analytic $\mathbf{x}_{2}$-graph around every point where $\partial_{2}$ doesn't belong to the tangent plane, i.e. $\left\langle N,\partial_{2} \right\rangle \neq 0$. But

$$\left\langle N,\partial_{2} \right\rangle = 0 \Longleftrightarrow \cosh{(v)} e^{-u} = 0,$$
which is impossible. Thus, $\mathcal{S}$ is a local analytic $\mathbf{x}_{2}$-graph around every point. Then, let's take a look to sections of the surface $\mathcal{S}$ by planes $\{ \mathbf{x}_{3} = \mbox{cte} \}$ : on the plane $\{ \mathbf{x}_{3} = \lambda \}$, we get the curve

$$c_{\lambda}(t) = \begin{pmatrix}
-\dfrac{\tanh{(t)}}{2} \left( 1+e^{-2\lambda} \cosh^{2}{(t)} \right) \\
\dfrac{e^{2 \lambda}}{4 \cosh^{2}{(t)}} - \dfrac{\lambda}{2} + \dfrac{\ln{(\cosh{(t)})}}{2} - \dfrac{\cosh{(2t)}}{4}
\end{pmatrix} := \begin{pmatrix}
{\mathbf{x}_{1}}_{\lambda}(t) \\
{\mathbf{x}_{2}}_{\lambda}(t)
\end{pmatrix}.$$

\begin{figure}[H]
    \centering
        \includegraphics[height=6cm,width=9cm]{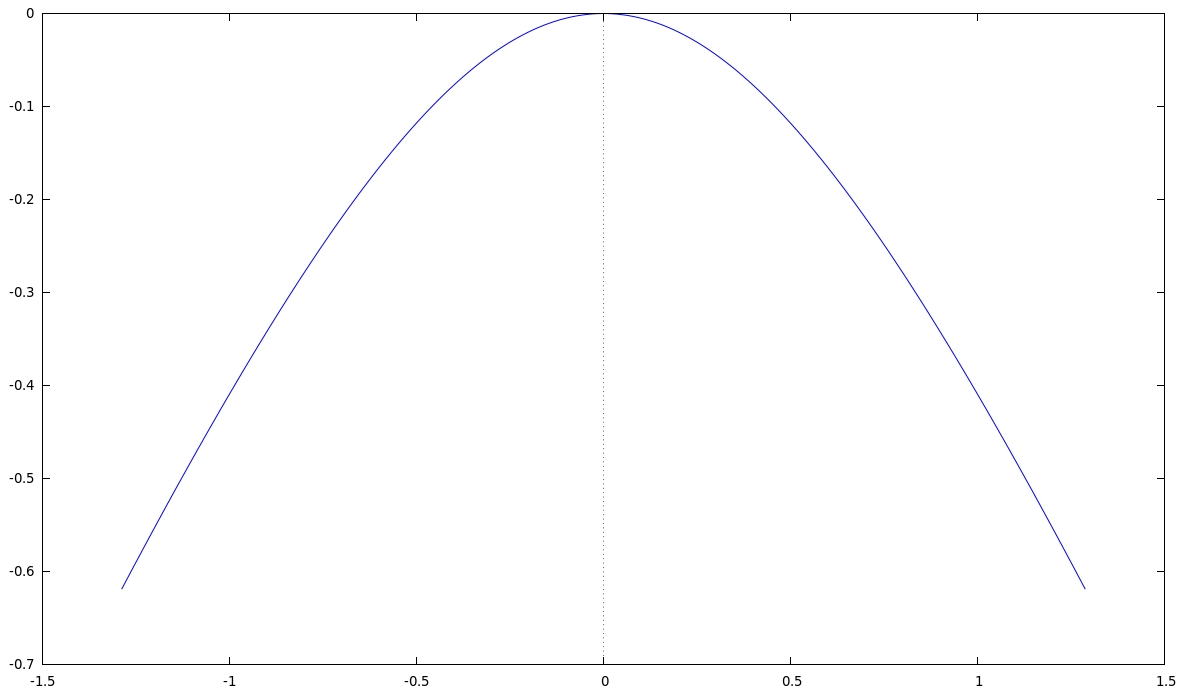}
    \caption{Section with $\{ \mathbf{x}_{3} = 0 \}$ with Maxima®}
\end{figure}

\begin{figure}[H]
    \centering
        \includegraphics[height=6cm,width=9cm]{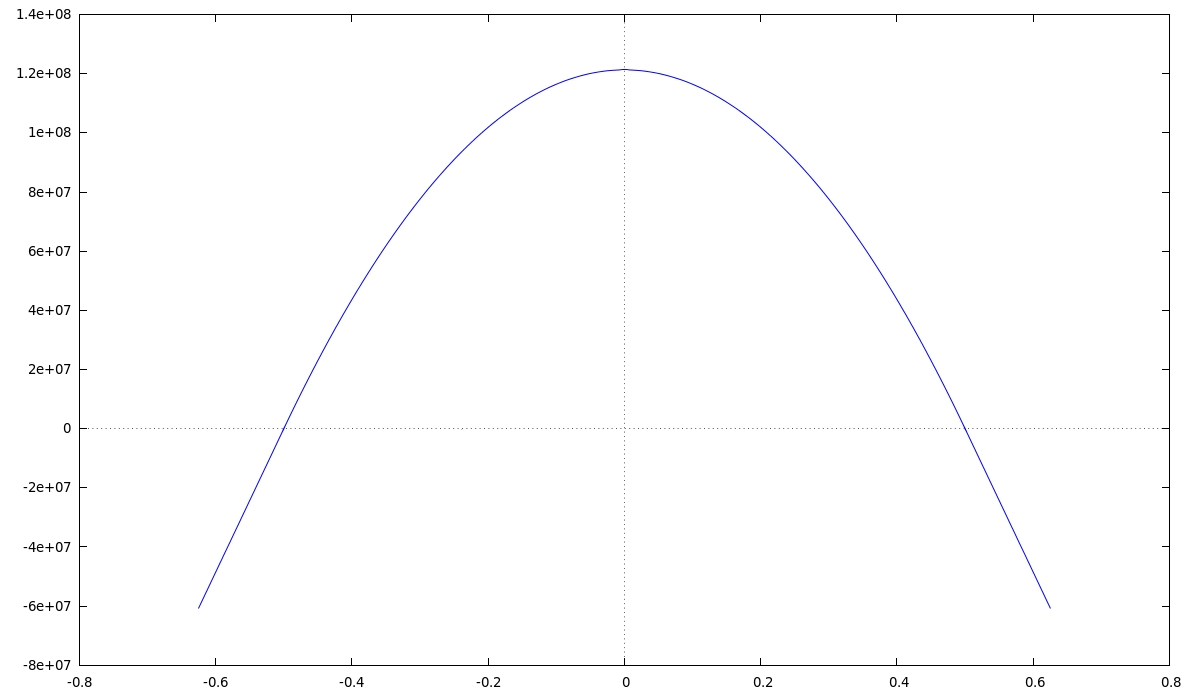}
    \caption{Section with $\{ \mathbf{x}_{3} = 10 \}$ with Maxima®}
\end{figure}

\begin{figure}[H]
    \centering
        \includegraphics[height=6cm,width=9cm]{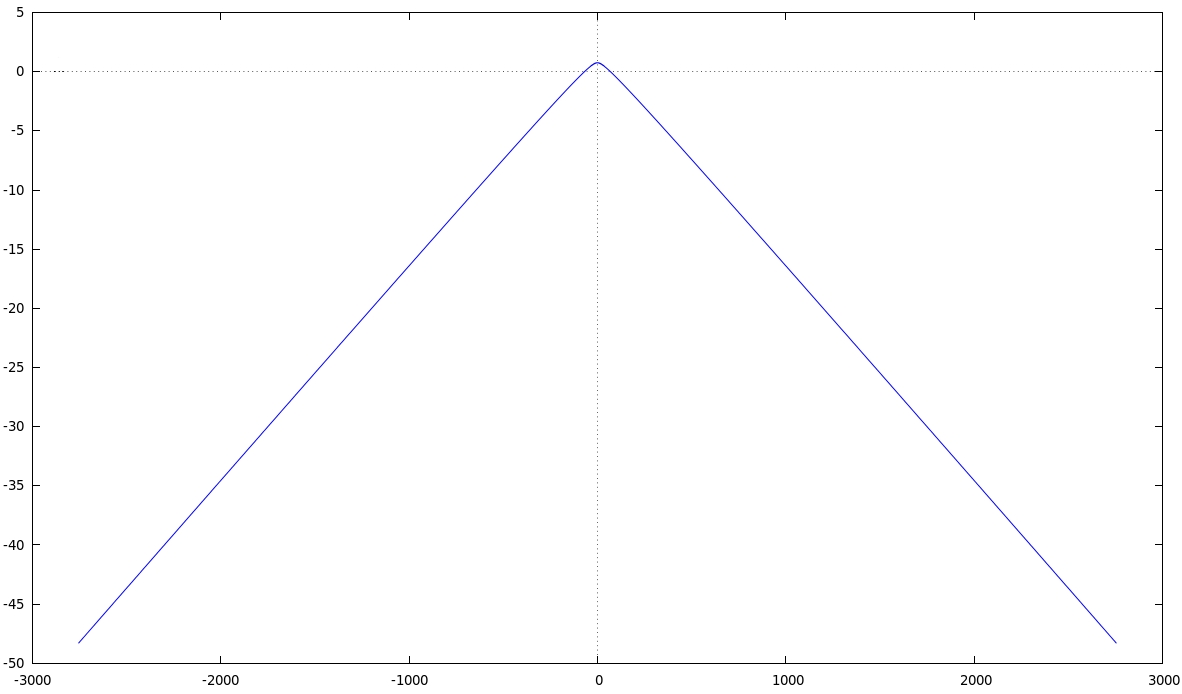}
    \caption{Section with $\{ \mathbf{x}_{3} = -2 \}$ with Maxima®}
\end{figure}

Then,

$${\mathbf{x}_{1}}_{\lambda}'(t) = \dfrac{\tanh^{2}{(t)}-1}{2} - \dfrac{e^{-2\lambda}}{2}(\cosh^{2}{(t)} + \sinh^{2}{(t)} ) < 0$$
for all $t \in \mathbb{R}$. Thus, the curves are injective, so the surface $\mathcal{S}$ is embedded. Moreover, by the implicit function theorem, we deduce that for every $\lambda \in \mathbb{R}$, there exists a function $f_{\lambda}$ such that ${\mathbf{x}_{2}}_{\lambda} = f_{\lambda}({\mathbf{x}_{1}}_{\lambda})$. Because the function ${\mathbf{x}_{1}}_{\lambda}$ is a decreasing diffeomorphism of $\mathbb{R}$, the function $f_{\lambda}$ is defined on the whole $\mathbb{R}$. This calculus is valid for all $\lambda \in \mathbb{R}$, so there exists a function $f : \mathbb{R}^{2} \longrightarrow \mathbb{R}$ such that $\mathbf{x}_{2} = f(\mathbf{x}_{1},\mathbf{x}_{3})$.

Finally, this function $f$ coincides around every point with the local $C^{\infty}$-functions which give the local graphs, and so $f$ is $C^{\infty}$.

\end{proof}

As a conclusion, we can notice that, for a fixed $\mathbf{x}_{3}$ :

\begin{itemize}
\item[\textbullet] When $\mathbf{x}_{1} \longrightarrow +\infty$, $\mathbf{x}_{2} \approx - \mathbf{x}_{1} e^{2 \mathbf{x}_{3}}$ ;
\item[\textbullet] When $\mathbf{x}_{1} \longrightarrow -\infty$, $\mathbf{x}_{2} \approx \mathbf{x}_{1} e^{2 \mathbf{x}_{3}}$.
\end{itemize}

\begin{figure}[H]
    \centering
        \includegraphics[height=10cm,width=15cm]{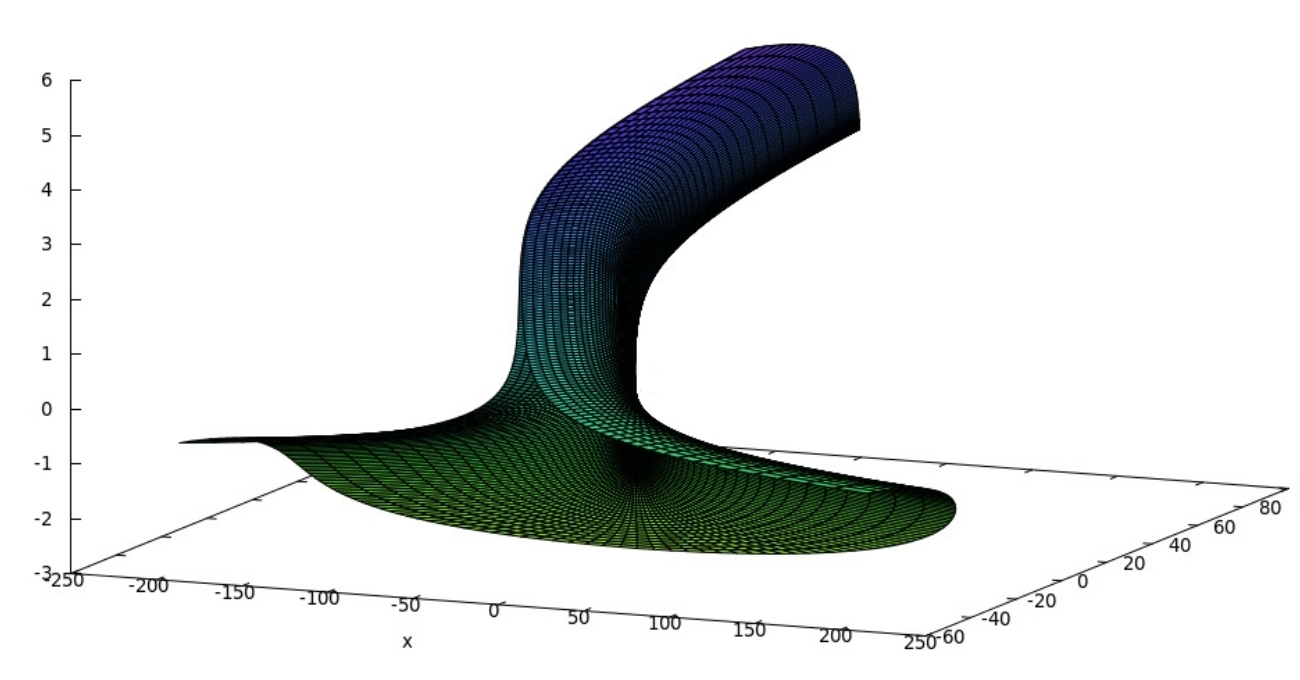}
    \caption{The surface $\mathcal{S}$ with Maxima®}
\end{figure}

\bibliographystyle{plain}
\bibliography{Biblio}
\addcontentsline{toc}{section}{References}

\textsc{Institut Elie Cartan de Lorraine, Université de Lorraine, Nancy, France.}

\textit{Current address :} Institut Elie Cartan de Lorraine, Université de Lorraine, B.P. 70239, 54506 Vandoeuvre-lès-Nancy Cedex, France

\textit{E-mail address :} christophe.desmonts@univ-lorraine.fr